\documentclass [12pt]{article}
\usepackage[T2A]{fontenc}
\usepackage{amsfonts,amssymb,amsthm}
\usepackage{graphicx}
\usepackage{amsbsy}
\usepackage{latexsym}
\usepackage{amsmath}
\usepackage{euscript}
\usepackage{setspace}
\newcommand{\R}{{\rm I}\kern-0.18em{\rm R}}

\input{amssym.def}

\newtheorem {theoremF}{Theorem}
\newtheorem {lemmaF}[theoremF]{Lemma}
\newtheorem {definF}{Definition}

\usepackage{srcltx}

\newcommand{\Rg}       {{\hbox{I\kern-.22em\hbox{R}}}}
\newcommand{\Pg}       {{\hbox{I\kern-.22em\hbox{P}}}}
\newcommand{\Eg}       {{\hbox{I\kern-.22em\hbox{E}}}}

\newcommand{\eps}{\varepsilon}

\newtheorem{theoremK}{Theorem}
\newtheorem{propositionK}[theoremK]{Proposition}

\theoremstyle{definition}

\begin{document}

\setcounter{section}{0}

\setcounter{equation}{0}



\begin{center}
{\LARGE \bf Ergodic properties of Fractional Stochastic Burgers
Equation}

\vspace{5mm}



{\bf Zdzis{\l}aw  Brze{\'z}niak, Latifa  Debbi and Ben Goldys}\\

\vspace{3mm}

Department of Mathematics, University of York, Heslington, York YO10
5DD, UK.\\
E-mail: zb500@york.ac.uk

\vspace{3mm}

Department Mathematics and Information Technology Chair of Applied
mathematics, Montan University 8700 Leoben,
Franz-Josef-Strasse 18 Austria.\\
E-mail:ldebbi@yahoo.fr and latifa.debbi@unileoben.ac.at

\vspace{3mm}

School of Mathematics, The University of New South Wales, Sydney NSW 2052, Australia.\\
E-mail: b.goldys@unsw.edu.au

 \vspace{2mm}


\vspace{1mm}

\end{center}

\begin{abstract}
We prove the existence and uniqueness of invariant measures for the
fractional stochastic Burgers equation (FSBE) driven by fractional
power of the Laplacian and space-time white noise. We show also that
the transition measures of the solution converge to the invariant
measure in the norm of total variation. To this end we show first
two results which are of independent interest: that the semigroup
corresponding to the solution of the FSBE is strong Feller and
irreducible.

\vspace{2mm}



\end{abstract}

\section*{1 Introduction}

\setcounter{equation}{0}\setcounter{section}{1}

Dynamics of many complex phenomena is driven by nonlocal
interactions. Such problems are usually modeled by means of
evolution equations including fractional powers of differential
operators or more general pseudo-differential operators. For
example, such equations arise in the theory of the quasi-geostrophic
flows,  the fast rotating fluids,  the dynamic of the
frontogenesis\footnote{The frontogenesis is the terminology used by
atmosphere scientists for describing the formation in finite time of
a discontinuous temperature front.}  in meteorology,  the diffusions
in fractal or disordered medium, the pollution problems, the
mathematical finance and the transport problems, see for a short
list e.g. \cite{Biler-Funaki-98, Caffarelli-2009,
Caffarelli-Vasseur2010, Constantin-chaae-wu-Arxiv10,
Kiselev-Nazarov-Fractal-Burgers-08, SugKak, Pablo-VazquezA-Arxiv,
Sug}  and the references therein. In \cite{SugKak, Sug} Kakutani and
Sugimoto studied the wave propagation in complex solids, especially,
viscoelastic materials (for example Polymers). They proved that the
viscoelasticity affects the behavior of the wave.  In particular,
they showed that the relaxation function has the form $ k(t)=
ct^{-\nu}, \; 0<\nu<1 ,\;\; c \in \mathbb{R}$, instead of the
exponential form known in the standard models (Maxwell-Voigt and
Voigt). This polynomial relaxation, called slow relaxation,  is due
to the non uniformity of the material. The main reason of this non
uniformity  is  the accumulation of several relaxations in different
scales. The far field is then described by a Burgers equation with
the leading operator $ (-\Delta)^\frac{1+\nu}{2}$ instead of the
Laplacian:
$$\partial_tu = -(-\Delta)^\frac{1+\nu}{2}u + \partial_xu^2 .$$
The above equation also describes the far-field evolution of
acoustic waves propagating in a gas-filled tube with a boundary
layer and has been used to study the acoustic waves in tunnels
during the passage of the trains\cite{Sug}. Indeed, in the last case
the geometrical configurations can yield a memory effect and other
types of resonance phenomena.

Frequently, lack of information about properties of the system makes
it natural to introduce stochastic models. Moreover, the stochastic
models are also powerful tools in the study of stability of
deterministic systems under small perturbations. The stochastic
character is visible when the initial data is  random and/or when
the coefficients are random.

Stochastic Partial Differential Equations play an essential role in
the mathematical modeling of many physical phenomena. These
equations are not only generalizations of the deterministic cases,
but they lead to new and important phenomena as well. E.g. Crauel
and Flandoli in \cite{Crauel-Flandoli98} showed that the
deterministic pitchfork bifurcation disappears as soon as an
additive white noise of arbitrarily small intensity is incorporated
the model. In \cite{Hairer-MattinglyErgodicity-06}, Hairer and
Mattingly characterized the class of noises for which the 2
dimensional stochastic Navier-Stokes equation is ergodic. In recent
series of papers and lectures, Flandoli and his Co-authors proved
that for several examples of deterministic partial differential
equations which are illposedness a suitable random noise can restore
the illposedness see e.g. \cite{Barbato-Flandoli- Morandin-2010,
DaPrato-Flandoli-2010, Flandoli-10, Flandoli-etal-transportEq-10}.

In particular, the stochastic Burgers equation (briefly SBE) emerges
in the modeling of many phenomena, such as in hydrodynamics, in
cosmology and in turbulence see for short list e.g.
\cite{Albv-Flando-Sinai-08, DapratoDebussche-Temam-94,
Neate-Truman08} and the references therein. While the physical
models leading to Burgers equation are simple, the mathematical
study is quite difficult and complex. The nonlinear term in this
equation comes from kinematical considerations, hence it cannot be
replaced by some simplifications or modifications. Let us also
denote that recently some other generalizations of stochastic
Burgers equation have also been investigated see e.g
\cite{Rockner-BS-2006, Truman-Wu06}.

The aim of this paper is to study the ergodic properties of the
solution of the fractional stochastic Burgers equation. The ergodic
properties of several stochastic partial differential equations have
been extensively studied, e.g.
\cite{DaPrato-Debussche-Differentiability-98,
Daprato-Elworthy-zabczyk-95, DaPrato-Gatarek-95, DaPrato-Zabczyk-96,
Gatarek-Goldys-Inv-Meas-97, Goldys-Maslowski-05,
Hairer-MattinglyErgodicity-06, Peszat-Zabc-strong-Feller-95,
Shirikyan-06}. These equations do not recover the FSBE.  More
precisely, we are interested in the fractional stochastic Burgers
equation (FSBE) given by:

\begin{equation}\label{Eq-Brut-FSBE}
\partial_tu= -(-\Delta)^\frac\alpha2 u + \partial_xu^2 +
g(u)\partial^2_{tx}W_{tx},\quad t>0,\; x\in (0, 1),
\end{equation}
with the boundary conditions,
\begin{equation}\label{Eq-Brut-FSBE-BC}
u(t, 0)= u(t, 1)=0,\quad t>0,
\end{equation}
and the initial condition
\begin{equation}\label{Eq-Brut-FSBE-IC}
u(0,x) = u_0(x),\quad x\in[0,1]
\end{equation}
where $u_0 \in L^2(0, 1)$. We will denote by $A$ the negative
Dirichlet Laplacian in the space $H=L^2(0,1)$, that is
\[A=-\Delta,\quad D(A)=H^{2,2}(0,1)\cap H^{1,2}_0(0,1).\]
Then the fractional power $(-\Delta)^{\frac{\alpha}{2}}$ is defined
as a fractional power of the operator $A$ (see \cite[pp.
72-73]{Pazy_1983}):
\begin{equation}\label{Eq-Fract-Pazy}
(-\Delta)^\frac\alpha2:=A_\alpha u:=\frac{\sin \frac{\alpha \pi}{2}
}{\pi}\int_0^\infty t^{\frac{\alpha}{2}-1}A(tI+A)^{-1}dt.
\end{equation}
\noindent The diffusion coefficient $ g $ is a bounded and Lipschitz
continuous map from $\mathbb{R}$ to $\mathbb{R}$ and
$\partial^2_{tx}W_{tx}$ stands for the space-time white noise i.e. $
\partial^2_{tx}W_{tx}$ is the distributional derivative of a mean zero
Gaussian field $\{ W(t, x), t\geq 0, x\in [0, 1]\}$ defined on a
filtered probability space $ (\Omega, \mathcal{F},\mathbb{F},
\mathbb{P}) $, where $\mathbb{F}=\{\mathcal{F}_t\}_{t\geq0}$,  such
that the covariance function is given by:
\begin{equation}\label{Eq-Cov-W}
\mathbb{E}[W(t, x)W(s, y)]= (t\wedge s)x\wedge y,\quad  t,s \geq
0,\,\, x, y \in [0, 1].
\end{equation}
It is  now customary,  see for instance \cite{DaPrato-Zbc-92} (and
\cite{Brz+Debbi_2007}),  to rewrite the problem
(\ref{Eq-Brut-FSBE}-\ref{Eq-Fract-Pazy})  as a stochastic evolution
problem in  the Hilbert space $ H= L^2(0, 1)$ in the following way
\begin{equation}\label{FSBE}
\Bigg\{
\begin{array}{lr}
du(t)= (-A_\alpha u(t) + Bu^2(t))dt + g(u(t))\,dW_t, \; t>0,\\
u(0) = u_0\in L^2(0, 1).
\end{array}
\end{equation}
Here   $B:=\frac{\partial}{\partial x}$ and  $ W:=\{ W(t), t\geq 0\}
$  is an $H$-cylindrical Wiener process  given by (with
$(e_j)_{j=1}^\infty $ is an ONB of $H$),
\begin{equation}\label{Seri-W}
 W(t):= \sum_{j=1}^\infty \beta_j(t)e_j,
\end{equation}
where $(\beta_j)_{j=1}^\infty $ is a  sequence of independent real
Brownian motions.

Let us recall the definition of a solution we will use throughout
the whole paper.

\begin{definF} Suppose that $ 1<
\alpha \leq 2$. An  $ \mathbb{F}$-adapted $L^2(0,1)$-valued
continuous process $u= (u(t), t\geq 0 )$ is said to be a mild
solution of equation (\ref{FSBE}) if for some
$p>\frac{2\alpha}{\alpha-1}$
\begin{equation}\label{cond-1}
\mathbb{E}\sup_{t \in [0, T]}|u(t)|_{L^2}^{p} < \infty, \quad
T>0,\end{equation} the function $(0,t)\ni s\to
S_\alpha(t-s)Bu^2(s)\in L^2(0,1)$ is Bochner integrable and for
every $t\geq 0$, the following identity holds $\mathbb P$-a.s. in
$L^2(0,1)$
\begin{equation}\label{Principal Eq in Burgers integ form}
u(t)= S_{\alpha}(t)u_{0}+ \int_{0}^{t}S_{\alpha}(t-s)Bu^{2}(s)\,ds +
\int_{0}^{t}S_{\alpha}(t-s)g(u(s))\,dW(s).
\end{equation}
\end{definF}

For the reader's convenience we recall the existence and uniqueness
result for the case $ u_0 \in L^2$. For the general case and proof
see \cite{Brz+Debbi_2007}.

\begin{theoremF}\label{th-Existence}
Assume that $g:\mathbb{R}\to \mathbb{R}$ is a bounded and Lipschitz
continuous function and that $\frac32< \alpha<2$. Then for every
$u_0 \in L^2(0,1)$  there exists a unique mild solution $u$ of
equation (\ref{FSBE}).
\end{theoremF}
In what follows the initial data $u_0 \in L^2(0,1)$ will usually be
denoted by $x$ and the unique mild solution $u$ of equation
(\ref{FSBE}) whose existence is guaranteed by Theorem
\ref{th-Existence} will be denote by $u(t,x)$, $t \geq 0$.

\noindent Using standard arguments, see for example
\cite{DaPrato-Zbc-92}, we can prove that the solution of equation
(\ref{FSBE}) is an $L^2$-valued strong-Markov and Feller process.
Let us define, the transition semigroup corresponding to the FSBE
\eqref{FSBE}, i.e.  a family $ \mathbf{U}=\big(U_t\big)_{t\geq 0} $
of bounded linear operators on the space of bounded Borel functions
$ \mathcal{B}_b(L^2)$ by
\begin{equation}\label{eqn-U}
  (U_t\varphi)(x)= \mathbb{E}[\varphi(u(t,x)],\quad x\in L^2(0,1),\;\; t\geq 0.
\end{equation}
Note that $\mathbf{U}$ is a semigroup (although not strongly
continuous) on $ \mathcal{B}_b(L^2(0,1))$. Let
\[\mu_t(x,B)=\mathbb P(u(t,x)\in B),\quad t\ge 0,x\in L^2(0,1), B\subset L^2(0,1),\,\,\mathrm{Borel},\]
be the family of transition measures of the process $u$. Then,
clearly
\[U_t\phi(x)=\int_{L^2}\phi(y)\mu_t(x,dy),\quad \phi\in\mathcal B\left(L^2(0,1)\right).\]
 The main  aim of our article is  to study the ergodic properties of the
FSBE \eqref{FSBE}, or equivalently, the ergodic properties of the
semigroup $\mathbf{U}$.  Our main result is the following one.

\begin{theoremF}\label{Main-Theo}
(a) Assume that $g:\mathbb{R}\to \mathbb{R}$ is a bounded and
Lipschitz continuous function. Suppose that $ \frac32< \alpha<2$.
Then there exists an invariant measure $\mu$ for the transition
semigroup $ \mathbf{U}$ corresponding to the fractional
stochastic Burgers equation \eqref{FSBE}.\\
(b) Assume additionally that
 $$\inf_{x\in \mathbb{R}}\vert g(x)\vert >0.$$ Then the invariant
measure $\mu$ is unique and for every $x\in L^2(0,1)$
\begin{equation}\label{tv}
\lim_{t\to\infty}\left\|\mu_t(x,\cdot)-\mu(\cdot)\right\|_{TV}=0,
\end{equation}
where $\|\cdot\|_{TV}$ denotes the total variation norm.
\end{theoremF}
We note that \eqref{tv} trivially yields the strong mixing property,
that is,
\begin{equation}\label{Strong-Mixing}
\lim_{t\to\infty}\int_{L^2}\left|U_t\phi-\int_{L^2}\phi
d\mu\right|^2d\mu=0,
\end{equation}
for every $\phi:L^2(0, 1)\to\mathbb R$ such that
\[\int_{L^2}|\phi|^2d\mu<\infty.\]
The paper is organized as follows. Section 2 contains some auxiliary
facts and basic estimates that will be needed later. In the third
and fourth sections, we will study the regularity of the semigroup
corresponding to the solution of the FSBE. In particular, we will
prove the strong Feller property and the irreducibility in sections
three and four respectively.  In the last section we will prove the
existence of an invariant measure and, under stronger assumptions,
its uniqueness and the strong mixing property \eqref{Strong-Mixing}.

\section{Preliminaries and a priori estimates}
\label{sec-apriori}

Let $u=(u(t), t\geq 0)$ be a given process and let $ \gamma \in
\mathbb{R}_+$. We introduce the following stochastic process
$z_\gamma= (z_\gamma(t), t\geq 0)$ defined by
\begin{equation}\label{Eq-z-gamma-t}
z_\gamma(t):= \int_0^t
e^{-(t-s)(A_\alpha+\gamma)}g(u(s))\,dW(s),\quad t\geq 0.
\end{equation}
The process $ z_\gamma$ satisfies the following stochastic
differential equation
\begin{equation}\label{Eq-z-gamma-t-eq-diff}
\left\{
\begin{array}{rl}
dz_\gamma(t)&= -(A_\alpha+\gamma ) z_\gamma(t)dt+ g(u(t))\,dW(t), \;\;t \geq 0,\nonumber \\
z_\gamma(0) &= 0.
\end{array}\right.
\end{equation}
Let us remark that if a process $ u$ is a solution of \eqref{FSBE}
then the stochastic process $ v_\gamma:= u-z_\gamma$ satisfies
pathwise the following equation
\begin{equation}\label{Eq.Determ.v(t)}
\left\{
\begin{array}{rl}
\frac{d}{dt}v_\gamma(t)&= -A_{\alpha}v_\gamma(t) + B(v_\gamma(t)+z_\gamma(t))^{2}+ \gamma z_\gamma(t), \;\; t> 0,\\
v(0) &= u_0.
\end{array}\right.
\end{equation}

The following two results are the main tools allowing us to study
the long time behaviour of the norm $ |v_\gamma(t)|^2_{L^2}$. For $s
\in \mathbb{R}^+\setminus \mathbb{N}$ we will denote by $ H^{s, p} $
the fractional order Sobolev space (called also Bessel potential
spaces and sometimes Lebesgue Besov spaces),  defined by the complex
interpolation method, i.e.
\begin{equation}\label{eqn-sobolev_spaces}
H^{s, p}(0,1)=[H^{k, p}(0,1), H^{m,p}(0,1)]_{\vartheta},
\end{equation}
 where
$k, m \in \mathbb{N}, \vartheta \in (0,1)$, $k<m$, are chosen to
satisfy
\begin{equation}\label{eqn-2.1c}
 s=(1-\vartheta)k+\vartheta m. \end{equation}
These spaces coincide with the Sobolev spaces $ W^{s, p}(0,1)$ for
integer values of $s$ if $ 1<p<\infty $ and for all $s$ when $p=2$
see e.g. \cite[p 219]{Adams-Sbolev-spaces75}, \cite{R&S-96} and
\cite[p 310]{Triebel-interpolation-theo-78}. Moreover, under
condition \eqref{eqn-2.1c}, we have, see e.g. \cite[p 103, p 196 \&
p 336]{Triebel-interpolation-theo-78}
\begin{equation}\label{eqn-sobolev_spaces-domain-of- operator}
H^{s, p}(0,1)= D((-\Delta)^{\frac s2})= [L^{p}(0,1), H^{m,
p}(0,1)]_{\vartheta}.
\end{equation}
In what follows  by $H_0^{s,p}(0,1)$, $s\geq 0$, $p\in(1,\infty)$,
we will denote the closure of $C_0^\infty(0,1)$ in the Banach space
$H^{s,p}(0,1)$. It is well known, see e.g. \cite[Theorem
11.1]{LM-72-i} and \cite[Theorem 1.4.3.2,
p.317]{Triebel-interpolation-theo-78} that
$H_0^{s,p}(0,1)=H^{s,p}(0,1)$ iff $s\leq \frac1{p}$. The norms in
various Sobolev spaces $H^{s,p}(0,1)$ will be denoted by $\vert
\cdot\vert_{H^{s, p}}$. Similarly, the norm in the $L^p(0,1)$ space
will be denoted by $\vert \cdot\vert_{L^p}$.

\begin{lemmaF}\label{lem-3-3brzezniak-debbi}\cite[Lemma 3.4]{Brz+Debbi_2007}
There exist $ \nu_1>0$, $ q>2$,  $ s \in (\frac1q, \frac12)$ and $
C>0$, such that
\begin{eqnarray}\label{ineq-energy-2}
\nonumber \frac{d}{dt}|v_\gamma(t)|_{L^2}^{2} &\leq&
-\frac{\nu_1}{2} |v_\gamma(t)|^2_{H^{\frac\alpha2,2}_0}+C
|z_\gamma(t)|^{\frac{\alpha}{\alpha-1}}_{H^{s,
q}}|v_\gamma(t)|^2_{L^2}\\&+&
C|z_\gamma(t)|^4_{H^{s,q}}+C|z_\gamma(t)|^4_{H^{1-\frac\alpha2,2}}+\gamma^2
C|z_\gamma(t)|^2_{L^2}, \;t \geq 0.
\end{eqnarray}
\end{lemmaF}
\begin{proof}
The proof is similar to the proof of  \cite[Proposition
3.3]{Brz+Debbi_2007}.
\end{proof}

\begin{lemmaF}\label{lem-est-v-gamma}
Let $(v_\gamma(t))_{t\geq0}$ be the solution of the equation
\eqref{Eq.Determ.v(t)}. Then there exist $ \nu\geq 0, \;\; C>0$, $
s\in (\frac1q, \frac12)$, $ q\in (2,\infty)$ and $s_0\in (\frac12,
\frac\alpha2]$, such that
\begin{eqnarray}\label{Eq.est.diffv(t)}
\nonumber \frac{d}{dt}|v_\gamma(t)|_{L^2}^{2} &+&
\nu|v_\gamma(t)|^2_{H^{\frac\alpha2,2}_0} \leq
C|z_\gamma(t)|_{H^{1-\frac\alpha2,
2}}^{\frac{2\alpha}{\alpha-s_0}}\big|v_\gamma(t) \big|_{L^{2}}^{2}
+C|z_\gamma(t)|^4_{H^{s,q}}\nonumber\\
&+&C|z_\gamma(t)|^4_{H^{1-\frac\alpha2,2}}+C\gamma^2\big|z_\gamma
(t)\big|_{L^2}^{2} , \;t \geq 0.
\end{eqnarray}
\end{lemmaF}
\noindent Lemma \ref{lem-est-v-gamma} is an improved version of
Lemma \ref{lem-3-3brzezniak-debbi}. The proofs are only slightly
different from those in \cite{Brz+Debbi_2007}. They are based on the
Sobolev embedding theorems \cite[Theorem 1, Tr 6, 3.3.1 section
2.4.4, p 82 \& Proposition Tr 6, 2.3.5 section 2.1.2, p 14]{R&S-96}
as well as the following one which we recall for the readers
convenience.
\begin{theoremF}\label{TeormRS-Multip} \cite[Theorem 2, section
4.4.4, p.177 \& Proposition Tr 6, 2.3.5 section 2.1.2, p.14]{R&S-96}
Assume that $0<s_1\leq s_2$,
$\frac1{p}\leq \frac1{p_1}+\frac1{p_2}
$,
$s_1+s_2>\frac{n}{p_1}+\frac{n}{p_2}-n
$ and  if $s_1<s_2$, $ p\geq p_1 $. Assume also that
\begin{equation}\label{Eq-con-RS-16}
\frac{n}{p}-s_1=
\begin{cases} (\frac{n}{p_1}-s_1)_+ + (\frac{n}{p_2}-s_2)_+, &
\text{ if } \; \max_{i}(\frac{n}{p_i}-s_i)>0 , \cr
\max_{i}(\frac{n}{p_i}-s_i)>0, & \text{ otherwise}
\end{cases}
\end{equation}
Then, provided that $\{i\in\{1,2\}: s_i=\frac{n}{p_i}\, \text{ and }
p_i>1 \}=\emptyset $,

\[H^{s_1, p_1}(0, 1)\cdot H^{s_2, p_2}(0, 1) \hookrightarrow H^{s_1, p}(0, 1).\]
\end{theoremF}
\begin{propositionK}\label{Gprop-deriv_estimate}\cite[Proposition 3.3 ]{Brz+Debbi_2007}
Assume that $\beta\in (\frac12,1)$. Then there exists a constant
$C>0$ such that for each $u\in H_0^{\beta, 2}(0,1)$ and each $v\in
H^{1-\beta, 2}(0,1)$ the following inequality is satisfied
\begin{equation}\label{ineq-02}
\left|\int_0^1 u(x)Dv(x)\, dx \right| \leq  C|u|_{H_0^{\beta,2}}
|v|_{H^{1-\beta, 2}}.
\end{equation}
\end{propositionK}
The following result is a generalization of \cite[Lemma
5.1]{DaPrato-Gatarek-95}. In order to formulate and prove it we need
to recall some basic facts about  the It\^{o} integration in
martingale type-2 Banach spaces.  For more information, see
e.g.\cite{Brz-Gat-99, Brz_1997, Brz+vN_2003, Neerven-2010}. For
brevity reason and to focus on the stochastic integral required in
this paper, we restrict ourselves to the case when the separable
Banach space $ E$ is $L^q(O, \mathcal{O}, \nu)= L^q(O)$, where $ (O,
\mathcal{O}, \nu)$ is a $\sigma-$finite measure space. In this case
and for  $ 2\leq q<\infty $ the space $ E= L^q(O)$ is both a UMD and
type-2 (and hence, see \cite{Brz_1997} a martingale type 2). The
construction of an It\^o integral for such spaces is described in
the paper \cite{Brz_1997} by the first named author.
\begin{definF}\label{Def-radonifying-operator}
Let $H$ be a separable Hilbert space and let  $E$ be a separable
Banach space. A bounded linear operator $ L: H \rightarrow E$,  is
called $\gamma-$radonifying operator if and only if for some (or
equivalently for every ) orthonormal basis (ONB)
$(h_j)_{j=1}^\infty$ of $H$ and some (equivalently every)   sequence
$(\gamma_j)_{j=1}^\infty$  of iid $N(0,1)$ random variables defined
on a probability space $(\Omega', \mathcal{F}', \mathbb{P}')$.  we
have\footnote{Actually, in view of the It\^o-Nisio Theorem, the
class $R_\gamma(H, E)$, the exponent $2$ below can be replaced by
any $p\in (1,\infty)$.}
\begin{equation}\label{eq-def-norm-radonifying}
\Vert L\Vert_{R_\gamma(H, E)}^2 := \mathbb{E}'|\sum_{j=1}^\infty
\gamma_jLh_j|^2_E <\infty.
\end{equation}
The space of all $\gamma$-radonifying operators from $H$ to $E$ will
be denoted  by $R_\gamma(H, E)$.
\end{definF}
\noindent Let us mention here that for certain classes of the target
Banach  spaces, a complete characterization of the
$\gamma$-radonifying operators can be given in a purely
non-probabilistic terms. For example, if  $E$ is a Hilbert space,
then a bounded linear operator $L \in R_\gamma(H, E)$ if and only if
it is a Hilbert-Schmidt operator. In this case  $ \Vert
L\Vert_{R_\gamma(H, E)}= \Vert L\Vert_{HS}$.
 If, see \cite{Brz+vN_2003},  $ E= L^q(O,  \mathcal{O}, \nu)$, as described above,
 then  $L \in R_\gamma(H, E)$ if and only if there exists
a function $\kappa\in L^q(O, \mathcal{O}, \nu;H)=L^q(O,H)$ such that
for every $h\in H$, $(Lh)(x)= \langle \kappa(x),h\rangle$ for
$\nu$-a.a. $x\in O$. In particular, see also \cite[Proposition
13.7]{Neerven-2010},  if  $1< q<\infty$ and $ (h_j)_{j\in J}$ is a
fixed ONB of $ H$ then  $L \in R_\gamma(H, E)$ if and only if $
(\sum_{j\in J}|Lh_j|^2)^\frac12$ is summable in $ L^q(O)$. In the
former, respectively the latter case, $\Vert L\Vert_{R_\gamma(H,
L^q)}$ is equivalent to $\vert \kappa\vert_{L^q(O, \mathcal{O},
\nu;H)}$, resp.
 $|(\sum_{j\in J}|Lh_j|^2)^\frac12|_{L^q}$.

\noindent Let us now recall the fundamental property of such an
integral resembling the Burhkolder inequality for local martingales,
see \cite{Brz_1997, Dettw_1985, Ondr_2004}. There exists a constant
$C=C_{p,q}>0$  such that for every progressively measurable process
$  \Phi \in L^p(\Omega; L^2(0, T; R_\gamma(L^2, L^q))$, where $ 2
\leq q<\infty $ and  $ 1<p<\infty $, we have
\begin{equation}\label{Burkholder}
\mathbb{E}|\sup_{t\in [0,T]}\int_0^t\Phi(s)\,dW(s)|_{L^q}^p\leq
C_{p,q}
\mathbb{E}\big(\int_0^T\Vert\Phi(t)\Vert_{R_\gamma(L^2,L^q)}^2\,dt\big)^{\frac
p2}.
\end{equation}

\noindent Finally, let us recall that $A$ is a selfadjoint operator
with compact inverse $A^{-1}$ and that for every $\alpha\in\mathbb
R$
\[A^{\frac{\alpha}{2}}e_j=\lambda_j^{\frac{\alpha}{2}}e_j,\quad j\ge 1,\]
where
\[e_j(\xi)=\sqrt{2}\sin\pi j\xi,\quad \lambda_j=\pi^2j^2.\]
\begin{lemmaF}\label{lem-A-sigmaZ-gamma-epsilon-s-q}
For all numbers $\eps  >0$, $ \sigma<\frac{\alpha-1}2$, $p\geq 1$
and $ q > 1$ we can find a positive number  $ \gamma_0 $ such that
for any $\gamma
>\gamma_0$,
\begin{equation}\label{Eq-supmainestimEZgamma-s-q}
\mathbb{E}\big|z_\gamma(t)\big|^p_{H^{\sigma, q}}\leq \eps , \;\;
t\geq 0.
\end{equation}
\end{lemmaF}
\begin{proof}
By inequality \eqref{Burkholder} there exists a constant $C:=C_{p,
q}>0$, such that for every $t\geq 0$,
\begin{eqnarray}\label{Eq-est-z-gamma-R-sum-1}
&&\hspace{-10truecm}\lefteqn{\mathbb{E}\big|(-A)^{\sigma/2} z_\gamma(t)\big|^p_{L^q}}\\
&&\hspace{-8truecm}\lefteqn{=
\mathbb{E}\big|(-A)^{\sigma/2} \int_0^t
e^{-(t-s)(A_\alpha+\gamma)}g(u(s))\,dW(s)\big|_{L^q}^p} \nonumber\\
&&\hspace{-8truecm}\lefteqn{\leq C_{p, q} \mathbb{E}\Big(\int_0^t\|(-A)^{\frac\sigma2}
e^{-(t-s)(A_\alpha+\gamma)}g(u(s))\|_{R_\gamma(L^2,
L^q)}^2ds\Big)^{\frac p2}.}
\nonumber
\end{eqnarray}
Thanks to \cite{Brz+vN_2003}, see also
 \cite[Proposition 13.7]{Neerven-2010},  there exists a positive constant $C:=C_{q}>0$,
(we keep the same notation for all constants), such that for all
$r\ge 0 $
\begin{equation}\label{Eq-est-z-gamma-R-sum-by-e-j}
\|(-A)^{\frac\sigma2}
e^{-r(A_\alpha+\gamma)}g(u(s))\|_{R_\gamma(L^2, L^q)}\leq C_q
|(\sum_{j=1}^\infty |(-A)^{\frac\sigma2}
e^{-r(A_\alpha+\gamma)}g(u(s))e_j|^2)^\frac12|_{L^q}.
\end{equation}
By arguing as in \cite{Brz+Debbi_2007} around inequality (C.15),  in
particular using Proposition C.1.6 and Lemma 2.4, we infer that
\begin{equation}
|(\sum_{j=1}^\infty |(-A)^{\frac\sigma2}
e^{-r(-A_\alpha+\gamma)}g(u(s))e_j|^2)^\frac12|_{L^q}\leq
b_0(\sum_{k=1}^{\infty}\lambda_k^\frac \sigma2e^{-r(\lambda_k^\frac
\alpha2+\gamma)}), \; r>0,
\end{equation}
Inserting  the above inequality  in
\eqref{Eq-est-z-gamma-R-sum-by-e-j} and using
\eqref{Eq-est-z-gamma-R-sum-1} and the H\"older inequality for
series, we get
\begin{eqnarray}
\mathbb{E}\big|(-A)^{\sigma/2} z_\gamma(t)\big|^p_{L^q}&\leq& C_{p,
q, b_0} \Big(\int_0^t\big(\sum_{k=1}^{\infty}\lambda_k^\frac
\sigma2e^{-s(\lambda_k^\frac \alpha2+\gamma)}\big)^2ds\Big)^{\frac
p2}\nonumber\\
&\leq& C_{p, q, b_0}
\Big(\int_0^t\big(\sum_{k=1}^{\infty}\lambda_k^\sigma
e^{-s(\lambda_k^\frac
\alpha2+\gamma)}\big)\big(\sum_{k=1}^{\infty}e^{-s(\lambda_k^\frac
\alpha2+\gamma)}\big)ds\Big)^{\frac
p2}.\nonumber\\
\end{eqnarray}
It is easy to see that there exist $
\rho\in\left(\frac1\alpha,1\right) $, (which we will let later tend
to $ \frac1\alpha$),  and $ C_\rho$, such that
\[
\big(\sum_{k=1}^{\infty}e^{-s(\lambda_k^\frac
\alpha2+\gamma)}\big)\leq C_\rho s^{-\rho},\quad  \gamma>0.\]
 By the change of variables $\xi:= s(\lambda_k^{\frac\alpha2}+\gamma)$
and letting $ t $ go to infinity in the RHS of the inequality above,
we infer that for every $t>0$,
\begin{eqnarray}\label{eq-final-est-Z-gamma}
\mathbb{E}\big|(-A)^{\sigma/2} z_\gamma(t)\big|^p_{L^q}&\leq & C_{p,
q, b_0, \rho}
\Big(\sum_{k=1}^{+\infty}\frac{\lambda_k^{\sigma}}{(\lambda_k^{\frac\alpha2}+\gamma)^{1-\rho}}
\int_0^{+\infty}\xi^{-\rho}e^{-\xi} d\xi\Big)^{\frac p2}.\nonumber \\
\end{eqnarray}

The series on the RHS of \eqref{eq-final-est-Z-gamma} converges
provided $ \sigma \in [0, \frac{\alpha-1}{2})$. In order to get the
inequality \eqref{Eq-supmainestimEZgamma-s-q}, it is sufficient to
observe that in view of the Lebesgue dominated convergence theorem
\[\lim_{\gamma\to\infty}\sum_{k=1}^{+\infty}\frac{(k\pi)^{2\sigma}}{\big((k\pi)^{\alpha}+\gamma\big)^{1-\rho}}=0.\]
 This completes the proof of the Lemma.
 \end{proof}

 \subsection*{Proof of Lemma \ref{lem-est-v-gamma}}

We note first that for every $T>0$ and each $ z_\gamma \in
L^\infty([0,T], H^{s, q}(0, 1))$, $q>2$ and where $ s \in (\frac1q,
\frac12)$, the deterministic problem \eqref{Eq.Determ.v(t)} has a
unique solution $v_\gamma \in C(0, T; L^2(0, 1))\cap L^2(0, T;
H^{\frac\alpha2, 2}_0(0, 1)) $. Multiplying the both sides of
\eqref{Eq.Determ.v(t)} by the solution $ v_{\gamma}(t)$ and applying
Lemma III.1.2 from \cite{Temam-2001}, then we get
\begin{eqnarray}\label{ineq-energy}
2\frac{d}{dt}|v_\gamma(t)|_{L^2}^{2}&=& -\langle A_\alpha
v_\gamma(t), v_\gamma(t) \rangle + \langle B (v_\gamma^2(t)),
v_\gamma(t) \rangle
\\
&+ & 2\langle B (z_\gamma(t)v_\gamma(t)), v_\gamma(t) \rangle + \langle B z_\gamma^{2}(t), v_\gamma(t) \rangle + \gamma \langle
z_\gamma(t), v_\gamma(t) \rangle.
\nonumber
\end{eqnarray}
In what follows, the time argument is, for simplicity of notations,
omitted. It is easy to see that $\langle B (v_\gamma^2), v_\gamma
\rangle =0 $  and that
\begin{eqnarray}\label{Est-z-v-gamma}
\gamma\big| \langle z_\gamma, v_\gamma \rangle\big|&\leq&
\gamma^2\big|z_\gamma\big|_{L^2}^2
+\nu_1\big|v_\gamma\big|_{H_0^{\frac\alpha2, 2}}^2.
\end{eqnarray}
Moreover, since $\frac\alpha4 \in (\frac38,\frac12)$ and $
\lambda_1=\pi^2$, we obtain

\begin{equation} \label{ineq-alpha} \langle A_\alpha v_\gamma, v_\gamma \rangle = \langle
A^{\frac\alpha2} v_\gamma, v_\gamma \rangle = \vert A^{\frac\alpha4}
v_\gamma \vert_{L^2}^{2}\geq \pi^{2\alpha}
|v_\gamma|_{H_0^{\frac{\alpha}{2},2}}^{2}, \; v_\gamma \in
D(A_\alpha).
\end{equation}
Applying  Proposition \ref{Gprop-deriv_estimate}, with
$\beta=\alpha/2$, we infer that

\begin{eqnarray*}
|\langle B(z_\gamma v_\gamma), v_\gamma \rangle|&\leq& C
|v_\gamma|_{H^{\frac\alpha2,2}_{0}}|z_\gamma
v_\gamma|_{H^{1-\frac\alpha2,2}}.
\end{eqnarray*}
By Theorems  \ref{TeormRS-Multip}, then there exist $ C>0$, $ q'>1$
and $ s'> \max\{\frac1{q'}, 1-\frac\alpha2\} $, such that
\begin{eqnarray*}
|z_\gamma v_\gamma|_{H^{1-\frac\alpha2,2}}\leq C |z_\gamma
|_{H^{1-\frac\alpha2, 2}}|v_\gamma|_{H^{s', q'}}.
\end{eqnarray*}
Furthermore, using \cite[Theorem 1, Tr 6, 3.3.1 section 2.4.4, p 82
\& Proposition Tr 6, 2.3.5 section 2.1.2, p 14]{R&S-96}, the
following embedding
\begin{equation*}
H^{s_0, 2}(0,1)\hookrightarrow H^{s', q'}(0,1)
\end{equation*}
holds for all  $ s_0\geq s'+\frac12-\frac1{q'} $ and $ s_0>s'$.
Hence, there exists a constant $C>0$, such that
\begin{equation*}
\big|v_\gamma \big|_{H^{s', q' }}\leq C\big|v_\gamma \big|_{H^{s_0,
2}}.
\end{equation*}
Now we choose $ s_0 \leq \frac\alpha2$. This choice is possible
thanks to the conditions $ \alpha>\frac32$ and $ q'>1$. In fact,
thanks to the above conditions, the  following inequalities hold
\begin{equation*}
\max\{\frac1{q'}, 1-\frac\alpha2\}<s'<s_0\leq
\frac\alpha2-\frac12+\frac1{q'}
\end{equation*}
and
\begin{equation*}
\max\{\frac1{q'}, 1-\frac\alpha2\}<s'<
\frac\alpha2-\frac12+\frac1{q'}.
\end{equation*}
By interpolation, there exists a constant $ C>0$, such that
\begin{equation*}
\big|v_\gamma \big|_{H^{s_0, 2}}\leq C\big|v_\gamma
\big|_{L^{2}}^{1-2\frac{s_0}{\alpha}}\big|v_\gamma
\big|_{H_0^{\frac\alpha2, 2}}^{2\frac{s_0}{\alpha}},
\end{equation*}
and thereby
\begin{eqnarray*}
|\langle B(z_\gamma v_\gamma), v_\gamma \rangle|&\leq&
C|v_\gamma|_{H^{\frac\alpha2,
2}_{0}}^{2\frac{s_0}{\alpha}}|z_\gamma|_{H^{1-\frac\alpha2,
2}}\big|v_\gamma \big|_{L^{2}}^{1-2\frac{s_0}{\alpha}}.
\end{eqnarray*}
Invoking the classical Young inequality we find that for some
generic constant $C>0$
\begin{eqnarray}\label{ineq-proof-01}
|\langle B(z_\gamma v_\gamma), v_\gamma \rangle|&\leq&
\nu_2|v_\gamma|_{H^{\frac\alpha2, 2}_{0}}^{2}+
C|z_\gamma|_{H^{1-\frac\alpha2,
2}}^{\frac{2\alpha}{\alpha-s_0}}\big|v_\gamma \big|_{L^{2}}^{2}.
\end{eqnarray}
To estimate the term $ |\langle B(z_\gamma^2), v_\gamma \rangle|$,
we follow the same calculation from \cite{Brz+Debbi_2007}. Applying
again Proposition \ref{Gprop-deriv_estimate} with $\beta=\alpha/2$
we infer that
\begin{eqnarray*}
|\langle B(z_\gamma^2), v_\gamma \rangle|&\leq& C
|v_\gamma|_{H^{\frac\alpha2,
2}_{0}}|z_\gamma^2|_{H^{1-\frac\alpha2,2}}
\end{eqnarray*}
Arguing as before with the choice  $ \frac1{q}<s<\frac12, \; q>2$,
we get
\begin{eqnarray}\label{ineq-proof-02}
\nonumber {|\langle B(z_\gamma^2), v_\gamma \rangle|}&\leq&
C|v_\gamma|_{H^{\frac\alpha2,2}_{0}}|z_\gamma|_{H^{s,
q}}|z_\gamma|_{H^{1-\frac\alpha2,2}}
\\&\leq& \nu_3 |v_\gamma|^2_{H^{\frac\alpha2,2}_{0}} +
C|z_\gamma|^4_{H^{s, q}}+C|z_\gamma|^4_{H^{1-\frac\alpha2, 2}}.
\end{eqnarray}
Combining equality (\ref{ineq-energy}), inequalities
\eqref{Est-z-v-gamma} and (\ref{ineq-alpha}) with inequalities
(\ref{ineq-proof-01}) and (\ref{ineq-proof-02}) and choose $ \nu_1,
\nu_2$ and $ \nu_3$, such that $ -\frac{\pi^{2\alpha}}{2}
+\frac{\nu_1}{2}+\nu_2 +\frac{\nu_3}{2}= -\nu$, where $ \nu
>0$, we get \eqref{Eq.est.diffv(t)}.

\section{Strong Feller Property}
\label{sec-SFP} We recall first the Feller property of the
transition semigroup associated to the solution of equation
\eqref{FSBE}.
\begin{propositionK} The transition
semigroup $ \mathbf{U}$ corresponding to the fractional stochastic
Burgers equation \eqref{FSBE} is a Markov Feller semigroup. i.e.
\begin{equation}
U_{t}\phi \in C_b\left(L^2(0,1)\right),\quad  \phi \in
\mathcal{C}_b\left(L^2(0,1)\right).
\end{equation}
\end{propositionK}
\begin{proof}
Omitted.
\end{proof}

In order to show the uniqueness of the invariant measure we will
need stronger properties of the transition semigroup. Below, for the
reader's convenience we recall the concept of the strong Feller
property and topological irreducibility.

\begin{definF}
A Markov semigroup   $\mathbf{U}=\left(U_t\right)_{t\geq 0}$  on a
Polish space $X$   is said to be (topologically) irreducible if for
any $t>0$,
\begin{equation}\label{eqn-irreducible}
 U_{t}{1}_{\Gamma}(x)>0 , \;\; x\in X,
\end{equation}
for arbitrary non empty open set $ \Gamma \subset X $. It   is said
to have the strong Feller property if for any $t>0$
\begin{equation}\label{eqn-SFP}
U_{t}\phi \in C_b(X)  , \;\;   \phi \in \mathcal{B}_b(X),
\end{equation}
where $\mathcal B_b(X)$ denotes the space of bounded Borel functions
on $X$.
\end{definF}

In order to prove the strong Feller property of the Markov semigroup
$ \mathbf{U}$ corresponding to the fractional stochastic Burgers
equation \eqref{FSBE} we introduce a family of truncated systems. As
in \cite{Brz+Debbi_2007}, for each natural number $n$ we define a
map $\pi_n :L^2(0,1)\to L^2(0,1)$ by
\[\pi_n(v)=\left\{\begin{array}{lll}
v&\mathrm{if}&|v|_{L^2}\le n,\\
\frac{n}{|v|_{L^2}}&\mathrm{if}&|v|_{L^2}>n.
\end{array} v \in L^2(0, 1)\right.\]
Then for each $n\ge 1$ we can consider an equation

\begin{equation}\label{Eq-modified}
\left\{
\begin{array}{rl}
du_n(t)&= \left(-A_\alpha
u_n(t)+B\left(\pi_n\left(u_n(t)\right)\right)^2\right)dt+
g\left(u_n(t)\right)\,dW(t),  \\
u_n(0) &= u_0,
\end{array}\right.
\end{equation}
By  \cite{Brz+Debbi_2007},  for each $u_0=x\in L^2(0,1)$, equation
\eqref{Eq-modified} admits a unique mild solution $(u_n(t, x), t\geq
0)$. Let $\mathbf{U}^n= \big(U^n_t\big)_{t\geq 0}$ be transition
semigroup associated to the solution of equation
\eqref{Eq-modified}, i.e. defined by
\begin{equation}\label{eqn-U^n}
  (U^n_t\varphi)(x)= \mathbb{E}[\varphi(u_n(t, x)], \;\; \;\phi\in \mathcal{B}_b(L^2(0,1)),
  \;\; x\in L^2(0,1).
\end{equation}
\begin{lemmaF}\label{unf}
for every $n\ge 1$ the semigroup $\mathbf{U}^n$ is strongly Feller
on $H=L^2(0,1)$. More precisely, for every $R>0$, $t>0$ there exists
$C(R,t)>0$ such that for any $\phi\in \mathcal{B}_b(H)$
\[\left|(U_t^n\phi)(x)-(U_t^n\phi)(y)\right|\le C(R,t)|x-y|_{L^2},\quad |x|_{L^2}\le R,\,\,|y|_{L^2}\le R.\]
\end{lemmaF}
\begin{proof}
The proof is a straightforward modification of the proof of
analogous statement in \cite{DaPrato-Gatarek-95}, hence omitted.
\end{proof}
\begin{theoremF}\label{Theorem-SF}
Assume that $g:\mathbb{R}\to\mathbb{R}$ is a Lipschitz continuous
function such that  there exist $a_0, b_0>0 $ such that
\begin{equation}\label{eqn-ass-g}  |g(x)| \in [a_0, b_0], \;\; x\in\mathbb{R}.
 \end{equation}
 Then the Markov semigroup $ \mathbf{U}$ corresponding to the fractional stochastic Burgers equation \eqref{FSBE}
is strong Feller.
\end{theoremF}
\begin{proof}

For $x\in L^2(0,1)$, let $(u(t, x), t\geq 0)$, respectively $(u_n(t,
x), t\geq 0)$, be the unique mild solution  of equation
\eqref{FSBE}, respectively \eqref{Eq-modified},  with the initial
condition $u(0,x)=x$. For  $n \in \mathbb{N}$ and  $x\in L^2(0,1)$
let $\tau_n(x)$ be an $\mathbb{F}$-stopping time defined  by
$$ \tau_n(x):= \inf\{t\geq
0: |u(t,x)|_{L^2}\geq n\}. $$ Let us fix $t>0$. It follows from
\cite[inequality (3.18)]{Brz+Debbi_2007} that one can find a
constant $C(t)>0$ and a sequence $(q_n)_{n=1}^\infty$, independent
of $t$, such that $q_k\nearrow \infty$ and
\begin{equation}\label{ineq-3.15}
 \mathbb{P}(\tau_n(x)<t)\leq \frac{C(t)}{q_n}(1+\ln^+\vert x\vert_{L^2}), \;\; n\in \mathbb{N}.
\end{equation}
On the other hand, due to the uniqueness property of the solution $u
$ proved in \cite[section 3]{Brz+Debbi_2007},
 $u(t, x)=u_n(t, x)$ for $t<\tau_n(x)$. Therefore, for all $ x \in
L^2(0,1)$ and $ \phi \in \mathcal{B}_b(L^2(0,1))$
\begin{eqnarray}
|(U_t\phi)(x)-(U_t^n\phi)(x)|&= & |\mathbb{E}\phi\big(u(t,
x)\big)-\mathbb{E}\phi\big(u_n(t,
x)\big)|\nonumber\\
&=&|\mathbb{E}[\phi\big(u(t,
x)\big)-\phi\big(u_n(t, x)\big)]1_{\tau_n(x)<t}|\nonumber\\
&\leq & 2|\phi|_{L^\infty}P(\tau_n(x)<t).
\end{eqnarray}
Hence, in view of inequality \eqref{ineq-3.15}, we infer that
$U_t^n\phi$ converges to $U_t\phi$ uniformly on balls in $L^2(0,1)$.
Hence, by applying Lemma \ref{unf}, the proof of Theorem
\ref{Theorem-SF} is complete.
\end{proof}

\section{Irreducibility}
\label{sec-irreducibility}

The main result of this section is as follows.
\begin{theoremF}\label{Theorem-irreducibility}
Assume that $g:\mathbb{R}\to\mathbb{R}$ is a Lipschitz continuous
function such that condition \eqref{eqn-ass-g} holds.
 Then the Markov semigroup $ \mathbf{U}$ corresponding to the fractional stochastic Burgers equation \eqref{FSBE}
 is irreducible.
\end{theoremF}

\begin{proof}
For $x\in L^2(0,1)$, let $(u(t, x), t\geq 0)$, respectively $(u_n(t,
x), t\geq 0)$, be the unique mild solution  of equation
\eqref{Principal Eq in Burgers integ form}, respectively
\eqref{Eq-modified},  with the initial condition $u(0, x)=x$.
 Let us fix $\eps >0$,  $T>0 $, $t\in [0,
T]$ and  $x,y \in L^2(0,1) $. To prove the topological
irreducibility, it is enough to show that
\begin{equation}\label{ineq-probability>0}
\mathbb{P}( u(t, x)\in B(y, \eps ))>0.
\end{equation}
Let us first observe that because $D(A_\alpha)$ is a dense subspace
of $L^2(0,1)$, it is enough to prove \eqref{ineq-probability>0} for
$ y \in D(A_\alpha)$. In what follows we make this additional
assumption.

For  $ n\geq 1$ and $ 0<\tau<t$, let us  define an
$\mathbb{F}$-progressively measurable process
\begin{equation}\label{Eq-f-tau-n-s-y}
f^{\tau, n}(s, y):=\Big\{
\begin{array}{lr}
\big[\frac1{t-\tau}e^{-(s-\tau)A_\alpha}(y-u(\tau, x))-A_\alpha
y\big]
\mathrm{1}_{B(0, n)}(u(\tau, x)), \;\;\text{if} \;\; s> \tau, \\
0,\;\;\text{if} \;\; s\leq \tau.
\end{array}
\end{equation}
where $ B(0, n)$  is the closed ball in $ L^2(0,1)$ of radius $ n$
and center $ 0$.  Later on we will choose  $ \tau $ to be
sufficiently close to $ t$. Let us  consider the following
stochastic evolution equation:

\begin{equation}\label{FSBE-h-n}
\Big\{
\begin{array}{lr}
du^{\tau, n}(s)= (-A_\alpha u^{\tau, n}(s) + \frac\partial{\partial
x}(u^{\tau, n}(s))^2+ f^{\tau, n}(s, y))ds +
g(u^{\tau, n}(s))\,dW(s), \\
u^{\tau, n}(0) = x,
\end{array}
\end{equation}
Repeating the argument from \cite{Brz+Debbi_2007} we can prove that
the problem \eqref{FSBE-h-n} has a unique mild solution $(u^{\tau,
n}(s, x))$, $s\geq 0$, such that
\begin{equation}\label{ben0}
M:=\mathbb E\sup_{s\in [0,T]}\left|u^{\tau,n}(s)\right|^2_{L^2}<\infty.
\end{equation}
 Next we define an $\mathbb{F}$--progressively measurable process $ (\beta^{\tau,
n}(s))_{s\geq 0}$ by
\begin{equation*}
\beta^{\tau, n}(s):= (g(u^{\tau, n}(s, x)))^{-1}f^{\tau, n}(s, y),
\end{equation*}
where for $ u \in L^2(0,1)$,  $g(u)$ is the multiplication operator
on $ L^2(0, 1)$.  Since the range of $ |g|$ is a subset of $ [a_0,
b_0]$, $ g(u)$ is a linear isomorphism of $ L^2(0,1)$.

\noindent By employing the  Girsanov Theorem   we will prove that
the laws of the processes $ u(\cdot, x)$  and $u^{\tau, n}(\cdot, x)
$ are equivalent on $C([0, T], L^2(0,1))$. Indeed, since
\begin{eqnarray}\label{eq-Prob-Beta-s-n-as-finite}
\mathbb{E} e^{\frac{1}{2}\int_0^t\vert\beta^{\tau,
n}(s)\vert_{L^2}^2\,ds}<\infty,
\end{eqnarray}
\cite[Lemmata 10.14 and 10.15]{DaPrato-Zbc-92}, the process
$\hat{W}$ defined by
\begin{equation}\label{eq-W-and-W-hat}
\hat{W}(t):= W(t)+ \int_0^t\beta^{\tau, n}(s)ds, \;\; t\in [0, T]
\end{equation}
is an $L^2$-cylindrical Wiener process on a filtered probability
space $ (\Omega, \mathcal{F}, \mathbb{F}_T, \mathbb{P}_T^\ast)$,
where $\mathbb{F}_T:=\{\mathcal{F}_t\}_{t\in[0, T]}$ and the
probability measure $ \mathbb{P}_T^\ast$  is  equivalent to $
\mathbb{P}$ with the Radon-Nikodym derivative $\frac{d\mathbb{P}_T^\ast}{d\mathbb{P}}$ equal to
\begin{equation}\label{eq-Z-T}
Z_T:= \exp\Big(\int_0^T\langle \beta^{\tau, n}(s), dW(s)\rangle-
\frac12 \int_0^T|\beta^{\tau, n}(s)|_{L^2}^2ds\Big).
\end{equation}
Therefore, using standard arguments we find that the process
$u^{\tau, n}(\cdot, x)$ solves the equation on $[0, T]$,
\begin{eqnarray}\label{Eq-diff-u-tau-n}
\nonumber du^{\tau, n}(s, x)&= &(-A_\alpha u^{\tau, n}(s, x) +
\frac\partial{\partial x}(u^{\tau, n}(s, x))^2 + f^{\tau, n}(s,
y))\,ds \\
&+& g(u^{\tau, n}(s, x))d\hat{W}(s)
- g(u^{\tau, n}(s, x))(g(u^{\tau, n}(s, x)))^{-1}[f^{\tau, n}(s, y)]\,ds\nonumber\\
&= &(-A_\alpha u^{\tau, n}(s, x) + \frac\partial{\partial
x}(u^{\tau, n}(s, x))^2)ds + g(u^{\tau, n}(s, x))\,d\hat{W}(s).\nonumber
\end{eqnarray}
 It follows from the uniqueness in law of the solution of the
\eqref{FSBE} that the laws on  $ C([0, T], L^2(0,1))$ of   $
u(\cdot, x) $ under $\mathbb{P}$ and  of $ u^{\tau, n}(\cdot, x) $
under $\mathbb{P}_T^\ast$ are the same. Hence,
\begin{eqnarray}\label{Eq-Prob-u-and-u-tau-n}
\mathbb{P}(u(t, x)\in B(y, \eps )) = \mathbb{P}_T^\ast(u^{\tau,
n}(t, x)\in B(y, \eps )).
\end{eqnarray}
Therefore, in order to prove that \eqref{ineq-probability>0} holds,
it is enough to show that we can choose $n$ and $\tau$ such that
$\mathbb{P}_T^\ast(u^{\tau, n}(t, x)\in B(y, \eps ))>0$.

Since $u^{\tau, n}(\cdot, x) $ is a mild solution to
\eqref{FSBE-h-n}, we infer that
\begin{eqnarray*}
u^{\tau, n}(t, x)&=& e^{-(t-\tau)A_\alpha} u(\tau, x) +
\int_\tau^te^{-(t-s)A_\alpha}
\big(\frac\partial{\partial x}(u^{\tau, n}(s, x))^2+ f^{\tau, n}(s, y))\big)ds \nonumber\\
&+& \int_\tau^te^{-(t-s)A_\alpha}g(u^{\tau, n}(s, x))\,dW(s), \,\,\,t>\tau.\\
\end{eqnarray*}
Direct calculations based on the semigroup property of $
(e^{-tA_\alpha})$ and equality \eqref{Eq-f-tau-n-s-y},  yield, see
also  \cite{Peszat-Zabc-strong-Feller-95},
\begin{equation*}
e^{-(t-\tau)A_\alpha} u(\tau) + \int_\tau^te^{-(t-s)A_\alpha}
f^{\tau, n}(s, y))ds=y,\quad t>\tau.
\end{equation*}
Consequently,
\begin{equation*}
u^{\tau, n}(t, x)= y+\int_\tau^te^{-(t-s)A_\alpha}
\frac\partial{\partial x}(u^{\tau, n}(s, x))^2ds + \int_\tau^te^{-(t-s)A_\alpha}g(u^{\tau, n}(s, x))\,dW(s).
\end{equation*}
Hence
\begin{eqnarray}\label{Eq-est-prob-difference-u-tau-n-and-y}
\mathbb{P}\big( u^{\tau, n}(t, x) \in B(y, \eps) \big) &\geq & 1-
\mathbb{P}\big( \left|\int_\tau^te^{-(t-s)A_\alpha}
\frac\partial{\partial x}(u^{\tau,
n}(s))^2ds \right|_{L^2}>\frac\eps 2 \big)\nonumber \\
&-& \mathbb{P}\big( \left|\int_\tau^te^{-(t-s)A_\alpha}g(u^{\tau,
n}(s))\,dW(s)\right|_{L^2}>\frac\eps 2 \big).\;\;
\end{eqnarray}
Invoking the Chebyshev inequality, inequality \eqref{Burkholder},
the boundedness of the operator $g$ and  the inequality $ \vert|
AB\vert|_{HS}\leq \vert|A\vert|_{HS}|B|_{\mathcal{L}(H)}$, we infer
that
\begin{eqnarray*}
&&
\hspace{-7truecm}\lefteqn{\mathbb{P}\big( |\int_\tau^te^{-(t-s)A_\alpha}g(u^{\tau,
n}(s))\,dW(s)|_{L^2}>\frac\eps 2 \big)} \\
 &\hspace{-4truecm}\lefteqn{\leq \frac{2^p}{\eps ^p}\mathbb{E}|\int_\tau^te^{-(t-s)A_\alpha}g(u^{\tau, n}(s))\,dW(s)|^p_{L^2}
 }\\
&&\hspace{-6truecm}\lefteqn{\leq \frac{2^p}{\eps ^p}\mathbb{E}
\big(\int_\tau^t\vert|e^{-(t-s)A_\alpha}g(u^{\tau, n}(s))\vert|_{HS}^2ds\big)^{\frac p2} \leq  \frac{2^p}{\eps ^p} b_0^p
\big(\int_0^{t-\tau}\vert|e^{-rA_\alpha}\vert|_{HS}^2 ds\big)^{\frac
p2}.}
\end{eqnarray*}
Since $\int_0^{T}\vert|e^{-rA_\alpha}\vert|_{HS}^2 \, ds<\infty $,
we can choose $ \tau <t$ such that

\begin{eqnarray}\label{Eq-est-P-stoch-int}
\mathbb{P}\big( |\int_\tau^te^{-(t-s)A_\alpha}g(u^{\tau,
n}(s))\,dW(s)|_{L^2}>\frac\eps 2 \big) &\leq &\frac14.
\end{eqnarray}
Furthermore, by the Chebyshev inequality, \cite[Lemma
2.11]{Brz+Debbi_2007} and \eqref{ben0}, we obtain
\begin{equation}\label{ben1}
\begin{aligned}
\mathbb{P}\left(
\left|\int_\tau^te^{-(t-s)A_\alpha}\frac\partial{\partial
x}(u^{\tau, n}(s))^2ds \right|_{L^2}> \frac\eps 2\right)
&\le\frac{4}{\eps^2}\mathbb E\left|\int_\tau^te^{-(t-s)A_\alpha}\frac\partial{\partial x}(u^{\tau, n}(s))^2ds \right|_{L^2}^2\\
&\le
\frac{4C_\alpha}{\eps^2}(t-\tau)^{1-\frac{3}{2\alpha}}\mathbb E\sup_{s\le t}\left|u^{\tau,n}(s)\right|^2_{L^2}\\
&\le \frac{4MC_\alpha}{\eps^2}(t-\tau)^{1-\frac{3}{2\alpha}}.
\end{aligned}
\end{equation}
Therefore, there exists $\tau<t$ such that
\[\mathbb{P}\left( \left|\int_\tau^te^{-(t-s)A_\alpha}\frac\partial{\partial x}(u^{\tau, n}(s))^2ds \right|_{L^2}> \frac\eps 2\right)\le\frac{1}{4}\]
and combining this result with \eqref{Eq-est-P-stoch-int} we find
that
\[\mathbb P\left(u^{\tau,n}(t,x)\in B(y,\eps)\right)\ge\frac{1}{2}.\]
Thanks to the equivalence of $\mathbb{P}$ and $\mathbb{P}_T^\ast$,
we obtain
\begin{eqnarray*}
\mathbb{P}_T^\ast\big( u^{\tau, n}(t, x) \in B(y, \eps)\big)>0.
\end{eqnarray*}
and invoking identity \eqref{Eq-Prob-u-and-u-tau-n},  we find that
\begin{eqnarray*}
\mathbb{P}\big( u(t, x) \in B(y, \eps) \big)= \mathbb{P}_T^\ast\big(
u^{\tau, n}(t, x) \in B(y, \eps)\big)>0.
\end{eqnarray*}
This proves \eqref{ineq-probability>0} for $ y \in D(A_\alpha)$. As
we observed at the beginning of the proof this implies the general
case and therefore   the proof of Theorem
\ref{Theorem-irreducibility} is complete.
\end{proof}

\section{Existence and  properties of the invariant measure}
\label{sec-measure}

The main task in this section is  to prove the existence of an
invariant measure. Our main tool is the Krylov-Bogoliubov Theorem.
Our result here is

\begin{theoremF}\label{Theo-Exist-invariant_measure}
There exists at least one  invariant measure for the transition
semigroup $\mathbf U$, corresponding to the fractional stochastic
Burgers equation \eqref{FSBE}.
\end{theoremF}
To prove this theorem we start with some auxiliary lemmas.
\begin{lemmaF}\label{lem-bound-mean}
There exists $\gamma_1$ such that for every $\gamma> \gamma_1$ and
every  $\eps >0$ we can find $ M>0 $ such that
\begin{eqnarray}\label{ineq-bound-mean}
\frac1t\int_0^t\mathbb{P}\Big(|v_\gamma (s)|^2_{L^2}\geq
M\Big)\,ds<\eps ,\;\; t>0.
\end{eqnarray}
\end{lemmaF}

\begin{proof}
Let us fix $\eps >0$. Let $C$ denote a constant greater than the
third power of the maximum of the embedding constant of $
H^{1-\frac\alpha2,2}(0,1)\hookrightarrow L^{2}(0,1)$, the embedding
constant of $H^{\frac\alpha2-\frac12, 2}(0,1) \hookrightarrow
H^{1-\frac\alpha2, 2}(0,1)$ and the constant $ C$ from
\eqref{Eq.est.diffv(t)}. By Lemma
\ref{lem-A-sigmaZ-gamma-epsilon-s-q}, we can find $ \gamma_0>0$ such
that, for all $ \gamma_1\geq \gamma_0$
\begin{equation}\label{Eq-supmainestimEZgamma}
\max \Big[\sup_{[0,
+\infty)}\mathbb{E}\big|z_{\gamma_1}(t)\big|^{\frac{2\alpha}{\alpha-s_0}}_{H^{\frac\alpha2-\frac{1}{2},
2} },  \sup_{[0,
+\infty)}\mathbb{E}\big|z_{\gamma_1}(t)\big|^4_{H^{\frac\alpha2-\frac{1}{2},
2} }, \sup_{[0,
+\infty)}\mathbb{E}\big|z_{\gamma_1}(t)\big|^2_{H^{\frac\alpha2-\frac{1}{2},
2} }\Big]\leq \eps \frac{\nu}{4C}
\end{equation}
Let $ M>1$. We define a process  $ \zeta_{\gamma_1}$ by  $
\zeta_{\gamma_1}(t)= \log\big(\big|v_{\gamma_1} (t)\big|^2_{L^2}\vee
M\big)$, $ t\geq 0$.  Using \eqref{Eq.est.diffv(t)} and the weak
derivative of $\zeta_{\gamma_1}(t) $, we get
\begin{eqnarray}
\zeta_{\gamma_1}'(t)&= &\frac1{|v_{\gamma_1}
(t)|^2_{L^2}}\mathbf{1}_{\{|v_{\gamma_1} (t)|^2_{L^2}\geq M\}}\frac
d{dt}|v_{\gamma_1} (t)|^2_{L^2}\nonumber \\
&\leq &\frac1{|v_{\gamma_1} (t)|^2_{L^2}}\mathbf{1}_{\{|v_{\gamma_1}
(t)|^2_{L^2}\geq M\}}\Big(-\nu\big|v_{\gamma_1}
(t)\big|_{L^{2}}^{2}+C|z_{\gamma_1}(t)|_{H^{1-\frac\alpha2,
2}}^{\frac{2\alpha}{\alpha-s_0}}\big|v_{\gamma_1}(t) \big|_{L^{2}}^{2}\nonumber\\
&+&C|z_{\gamma_1}(t)|^4_{H^{s,q}}+
C|z_{\gamma_1}(t)|^4_{H^{1-\frac\alpha2,2}}+\gamma_1^2|z_{\gamma_1}(t)|_{L^{2}}^{2}
\Big).
\end{eqnarray}
Taking expectation of the both sides, we get
\begin{eqnarray}\label{Eq-prb-estima}
\big(\mathbb{E}\zeta_{\gamma_1}(t)-\mathbb{E}\zeta_{\gamma_1}(0)\big)&+
&\nu\int_0^t\mathbb{P}\Big(|v_{\gamma_1} (s)|^2_{L^2}\geq
M\Big)ds
\\
&\leq&
\int_0^t\Big(C\mathbb{E}|z_{\gamma_1}(s)|_{H^{1-\frac\alpha2,
2}}^{\frac{2\alpha}{\alpha-s_0}}+ \frac C{M}\mathbb{E}|z_{\gamma_1}(s)|^4_{H^{s,q}}
\nonumber\\
&+&  \frac
C{M}\mathbb{E}|z_{\gamma_1}(s)|^4_{H^{1-\frac\alpha2,2}}+\frac
{C{\gamma_1}^2}{M}\mathbb{E}|z_{\gamma_1} (s)|_{L^{2}}^{2} \Big)ds.\nonumber\\
&\leq&t\Big(C\sup_{(0,\infty)}\mathbb{E}|z_{\gamma_1}(s)|_{H^{1-\frac\alpha2,
2}}^{\frac{2\alpha}{\alpha-s_0}}
+\frac C{M}\sup_{(0,\infty)}\mathbb{E}|z_{\gamma_1}(s)|^4_{H^{s,q}}\nonumber\\
&+ & \frac
C{M}\sup_{(0,\infty)}\mathbb{E}|z_{\gamma_1}(s)|^4_{H^{1-\frac\alpha2,2}}+\frac
{C{\gamma_1}^2}{M}\sup_{(0,\infty)}\mathbb{E}|z_{\gamma_1}
(s)|_{L^{2}}^{2} \Big).
\nonumber
\end{eqnarray}
It is easy to see that if $ M\geq |v(0)|^2_{L^2}$, then
\[
\mathbb{E}\zeta_{\gamma_1}(t)-\mathbb{E}\zeta_{\gamma_1}(0)\geq
0,\quad t\geq 0.\]
 Hence
\begin{eqnarray}\label{Eq-estim-time-mean-Pv-gamma}
\frac1t\int_0^t\mathbb{P}\Big(|v_{\gamma_1} (s)|^2_{L^2}\geq
M\Big)ds &\leq&\frac{C}{M\nu}\Big[M
\sup_{(0,\infty)}\mathbb{E}|z_{\gamma_1}(s)|_{H^{1-\frac\alpha2,
2}}^{\frac{2\alpha}{\alpha-s_0}}
\\
+\sup_{(0,\infty)}\mathbb{E}|z_{\gamma_1}(s)|^4_{H^{s, q}} &+ &
\sup_{(0,\infty)}\mathbb{E}|z_{\gamma_1}(s)|^4_{H^{1-\frac\alpha2,2}}+\gamma_1^2 \sup_{(0,\infty)}\mathbb{E}|z_{\gamma_1}
(s)|_{L^{2}}^{2}\Big].\nonumber
\end{eqnarray}
Let us remark that thanks to the condition $ \alpha>\frac32$, we
have $ H^{\frac\alpha2-\frac12, 2}(0,1)\hookrightarrow
H^{1-\frac\alpha2, 2}(0,1)$, hence $
|z_{\gamma_1}(s)|_{H^{1-\frac\alpha2, 2}}\leq
C|z_{\gamma_1}(s)|_{H^{\frac\alpha2-\frac12, 2}}$. Consequently,
from the formulae \eqref{Eq-supmainestimEZgamma} and for the choice
$ \frac1{q}<s<\frac\alpha2-\frac12<\frac12, \; q>\frac2{\alpha-1}$,
Lemma \ref{lem-A-sigmaZ-gamma-epsilon-s-q} and thanks to the
condition $ M> \gamma^2$, then we get

\begin{eqnarray}
\frac1t\int_0^t\mathbb{P}\Big(|v_\gamma (s)|^2_{L^2}\geq
M\Big)ds<\eps .
\end{eqnarray}
\end{proof}
The following Lemma is a slight generalization of \cite[Lemma
2.11]{Brz+Debbi_2007}.

\begin{lemmaF}\label{lem-A-beta_seconterm} For every $T>0$  there exist $ C>0$ and $\theta >0$ such
that for all $ v \in C([0, T]; L^1(0, 1))$,
\begin{equation}\label{Eq-second-term-beta}
\int_0^t\big|(-A)^{\frac{2\alpha-3}{4}}
e^{-(t-s)A_\alpha}Bv(s)\big|_{L^2}ds \leq
Ct^\theta\big|v\big|_{L^\infty([0, t];L^1(0, 1)}^2, \;\; t\geq 0.
\end{equation}
\end{lemmaF}
\begin{proof}
Let us choose (and fix) $\beta >0$ and $ v \in C([0, T]; L^1(0,
1))$. Then  we have
\begin{eqnarray*}
&&\hspace{-8truecm}\lefteqn{\big|(-A)^\beta
e^{-(t-s)A_\alpha}Bv(s)\big|_{L^2}^2=
\sum_{k=1}^{+\infty}\langle(-A)^\beta
e^{-(t-s)A_\alpha}Bv(s), e_k\rangle^2} \nonumber\\
= \sum_{k=1}^{+\infty}\langle v(s), B e^{-(t-s)A_\alpha}(-A)^\beta
e_k\rangle^2
&\leq&\sum_{k=1}^{+\infty}(k\pi)^2\lambda_k^{2\beta}e^{-\lambda_k^{\frac\alpha2}(t-s)}
\langle v(s),
f_k\rangle^2, \nonumber\\
\end{eqnarray*}
where $ f_k(\cdot):= \cos(k\pi\cdot)$. Thanks to the inequality $
|\langle v, f_k\rangle|\leq |v|_{L^1}| f_k|_{L^\infty}$, we get
\begin{eqnarray*}
\int_0^t\big|(-A)^\beta e^{-(t-s)A_\alpha}Bv(s)\big|_{L^2}ds
&\leq&\int_0^t
\Big(\sum_{k=1}^{+\infty}(k\pi)^{4\beta+2}e^{-(k\pi)^{\alpha}(t-s)}\Big)^\frac12
|v(s)|_{L^1(0, 1)}ds \nonumber\\
&\leq&\sup_{[0, t]}|v(s)|_{L^1(0, 1)} \int_0^t
\Big(\sum_{k=1}^{+\infty}(k\pi)^{4\beta+2}e^{-(k\pi)^{\alpha}s}\Big)^\frac12ds. \nonumber\\
\end{eqnarray*}
Since   $ e^{-(k\pi)^{\alpha}s} \leq C_\sigma
s^{-\sigma}(k\pi)^{-\sigma\alpha} $ for  all $ s\geq 0$ and a
certain constant  $ C_\sigma>0$, we get

\begin{equation}\label{eq-secondterm-beta}
\int_0^t
\Big(\sum_{k=1}^{+\infty}(k\pi)^{4\beta+2}e^{-(k\pi)^{\alpha}s}\Big)^\frac12ds
\leq C_N
\Big(\sum_{k=1}^{+\infty}(k\pi)^{4\beta+2-\sigma\alpha}\Big)^\frac12
\int_0^ts^{-\frac \sigma2}\,ds.
\end{equation}
The series and the integral in the RHS of \eqref{eq-secondterm-beta}
converge provided $ \frac{4\beta+3}{\alpha}<\sigma<2$. Such a value
of $ \sigma$ exists provided $ \beta<\frac{2\alpha-3}{4}$. Moreover,
if $ \alpha>\frac32$, then we can choose $ 0\leq
\beta<\frac{2\alpha-3}{4}$, such that
\begin{equation}
\int_0^t\big|(-A)^\beta e^{-(t-s)A_\alpha}Bv(s)\big|_{L^2}ds \leq
Ct^{1-\frac \sigma2}\sup_{[0, t]}|v(s)|_{L^1}.
\end{equation}

To get the inequality \eqref{Eq-second-term-beta}, we take $
\theta:= 1-\frac \sigma2$ and $ \beta = \frac{2\alpha-3}{4}$.
\end{proof}

\subsection*{Proof of Theorem
\ref{Theo-Exist-invariant_measure}} We use the Krylov-Bogoliubov
method, in particular, we follow the proof in
\cite{DapratoDebussche-Temam-94, DaPrato-Gatarek-95}. Let us remark
that for all $ \theta >0$ the embedding $ H^{\theta,
2}(0,1)\hookrightarrow L^2(0,1)$ is compact, see e.g.
\cite{DapratoDebussche-Temam-94} and
\cite[p273]{DaPrato-Zabczyk-96}. Hence the ball $$ B_{H^{\theta,
2}}(0, M):=\{v\in H^{\theta, 2}(0,1), |v|_{H^{\theta, 2}}\leq M\}$$
is a compact subset of $ L^2(0,1)$.

\noindent In the first step, let us prove that for all $ \beta \in
(0, \frac{2\alpha-3}{4})$ and  $\eps
>0$ there exists $ M>0$, large enough such that  $ \frac1t\int_0^tP\Big(|(-A)^\beta
v_\gamma (s)|^2_{L^2}\geq M\Big)ds\leq \eps $, where $ v_\gamma$ is
the solution of \eqref{Eq.Determ.v(t)}. We have,
\begin{eqnarray}\label{Eq-v-gamma-integ}
v_\gamma(t+1)= e^{-A_\alpha}v_\gamma(t)&+ & \int_t^{t+1}e^{-A_\alpha
(t+1-s)}B(v_\gamma(s)+z_\gamma(s))^{2}ds\nonumber \\
&+& \gamma\int_t^{t+1}e^{-A_\alpha (t+1-s)} z_\gamma(s)ds.\;\; a.s.
\end{eqnarray}
Let $ 0< \beta <\frac{2\alpha-3}{4}$. Then
\begin{eqnarray}\label{Eq-v-gamma-integ-Abeta}
\big|(-A)^\beta v_\gamma(t+1)\big|_{L^2}&\leq
&\big|A_\alpha^{2\frac\beta\alpha}
e^{-A_\alpha}v_\gamma(t)\big|_{L^2}\nonumber
\\ &+& \int_t^{t+1}\big|(-A)^\beta e^{-A_\alpha
(t+1-s)}B(v_\gamma(s)+z_\gamma(s))^{2}\big|_{L^2}ds\nonumber \\
&+& \gamma\int_t^{t+1}\big|(-A)^\beta e^{-A_\alpha (t+1-s)}
z_\gamma(s)\big|_{L^2}ds\;\; a.s. \nonumber \\
\end{eqnarray}
We denote the three terms in the RHS of
\eqref{Eq-v-gamma-integ-Abeta} by $I_1, I_2, I_3$ respectively. It
is easy to see, using the properties of the heat semigroup that
\begin{eqnarray}\label{Eq-v-gamma-integ-Abeta-term-1}
I_1:= \big|A_\alpha^{2\frac\beta\alpha}
e^{-A_\alpha}v_\gamma(t)\big|_{L^2} \leq
C\big|v_\gamma(t)\big|_{L^2} \; \;\; \text{and } \; \;\; I_3
\leq \gamma\int_t^{t+1}\big|z_\gamma(s)\big|_{H^{2\beta, 2}}ds.\nonumber\\
\end{eqnarray}
For the second term, arguing as in the proof of Lemma
\ref{lem-A-beta_seconterm}, we get for some  $ 0<\eta<1$,
\begin{eqnarray}\label{Eq-v-gamma-integ-Abeta-term-2}
I_2 &\leq &\int_t^{t+1}\big|(-A)^\beta e^{-A_\alpha
(t+1-s)}B(v_\gamma(s))^{2}\big|_{L^2}ds\nonumber \\
&+& 2\int_t^{t+1}\big|(-A)^\beta e^{-A_\alpha
(t+1-s)}B(v_\gamma(s)z_\gamma(s))\big|_{L^2}ds\nonumber \\
&+& \int_t^{t+1}\big|(-A)^\beta e^{-A_\alpha
(t+1-s)}B(z_\gamma(s))^{2}\big|_{L^2}ds\nonumber \\
&\leq &\sup_{0\leq r\leq 1}\big|v_\gamma(t+r)\big|^{2}_{L^2} +
2\int_t^{t+1}(t+1-s)^{-\eta}\big|v_\gamma(s)\big|_{L^2}\big|z_\gamma(s))\big|_{L^2}ds \nonumber \\
&+& \int_t^{t+1}(t+1-s)^{- \eta}\big|
z_\gamma(s)\big|^{2}_{L^2}ds \nonumber \\
&\leq &c_\eta\sup_{0\leq r\leq 1}\big|v_\gamma(t+r)\big|^{2}_{L^2}+
2\int_t^{t+1}(t+1-s)^{-\eta}\big|
z_\gamma(s)\big|^{2}_{L^2}ds. \nonumber \\
\end{eqnarray}

By subsisting  \eqref{Eq-v-gamma-integ-Abeta-term-1}-
\eqref{Eq-v-gamma-integ-Abeta-term-2} on
\eqref{Eq-v-gamma-integ-Abeta}, we get
\begin{eqnarray}
\big|(-A)^\beta v_\gamma(t+1)\big|_{L^2}&\leq&
C\big|v_\gamma(t)\big|_{L^2}+
c_\eta\sup_{0\leq r\leq 1}\big|v_\gamma(t+r)\big|^{2}_{L^2}\nonumber \\
&+& 2\int_t^{t+1}(t+1-s)^{-\eta}\big| z_\gamma(s)\big|^{2}_{L^2}ds+
\gamma\int_t^{t+1}\big|z_\gamma(s)\big|_{H^{2\beta, 2}}ds
.\;\; a.s.\nonumber \\
\end{eqnarray}
Hence
\begin{eqnarray}\label{eq-est-sobolev-greater-M-BL}
\frac1T\int_0^T\mathbb{P}\Big( &|&(-A)^\beta
v_\gamma(t+1)|^{2}_{L^2}>M \Big)dt \leq \frac1T\int_0^T
\mathbb{P}\Big( \big|v_\gamma(t)\big|^{2}_{L^2}>\sqrt{\frac
M{4C}}\Big)dt\nonumber
\\
&+& \frac1T\int_0^T \mathbb{P}\Big(\sup_{0\leq r\leq
1}\big|v_\gamma(t+r)\big|^{4}_{L^2}>\frac M{4c_\eta} \Big)dt\nonumber \\
&+& \frac1T\int_0^T \mathbb{P}\Big(\int_t^{t+1}(t+1-s)^{-\eta}\big|
z_\gamma(s)\big|^{2}_{L^2}ds>\frac M{8C} \Big)dt\nonumber \\
&+&
\frac1T\int_0^T\mathbb{P}\Big(\int_t^{t+1}\big|z_\gamma(s)\big|_{H^{2\beta,
2}}ds \Big) >\frac M {4\gamma}\Big)dt.
\end{eqnarray}
From Lemma \ref{lem-bound-mean}, there  exists $ M_1>0$, such that $
\forall M> M_1$, we have
\begin{eqnarray}\label{eq-est-sobolev-1}
\frac1T\int_0^T \mathbb{P}\Big(
\big|v_\gamma(t)\big|^{2}_{L^2}>\sqrt{\frac M{4C}}\Big)dt<\frac
\eps4.
\end{eqnarray}
By Chebyshev inequality and Lemma
\ref{lem-A-sigmaZ-gamma-epsilon-s-q} and considering the choice $
M>\max\{M_1, 16\gamma, 32C\} $ and  the condition $\alpha<2$ which
guaranty that $ \beta< \frac{2\alpha-3}{4}< \frac{\alpha-1}{4}$, we
infer
\begin{eqnarray}\label{eq-est-P-T-term-1-eps-4}
\frac1T\int_0^T &\mathbb{P}&\Big(\int_t^{t+1}(t+1-s)^{-\eta}\big|
z_\gamma(s)\big|^{2}_{L^2}ds>\frac M{8C} \Big)dt \nonumber\\
&\leq & \frac {8C} M \frac1T\int_0^T\Big(\int_t^{t+1}(t+1-s)^{-
\eta}\mathbb{E}\big|z_\gamma(s)\big|^{2}_{L^2}ds
\Big)dt <  \frac\eps4 \nonumber\\
\end{eqnarray}
and
\begin{eqnarray}\label{eq-est-P-T-term-eps-4}
\frac1T\int_0^T\mathbb{P}\Big(\int_t^{t+1}\big|z_\gamma(s)\big|_{H^{2\beta,
2}}ds \Big) >\frac M {4\gamma}\Big)dt &\leq & \frac {4\gamma} M
\frac1T\int_0^T\Big(\int_t^{t+1}\mathbb{E}\big|z_\gamma(s)\big|_{H^{2\beta,
2}}ds
\Big)dt\nonumber\\
&<& \frac\eps4.\nonumber\\
\end{eqnarray}
To estimate the second term in RHS of the inequality
\eqref{eq-est-sobolev-greater-M-BL}, we apply Gronwall Lemma to the
formulae \eqref{ineq-energy-2}, then we get
\begin{eqnarray}\label{eq-estima-energy-2}
\nonumber |v_\gamma(t+s)|_{L^2}^{2} &\leq& |v_\gamma(t)|_{L^2}^{2}
e^{C\int_0^s |z_\gamma(t+r)|^{\frac{\alpha}{\alpha-1}}_{H^{s,
q}}dr}+ \int_0^s
\Big[C|z_\gamma(t+r)|^4_{H^{s,q}}\\&+&C|z_\gamma(t+r)|^4_{H^{1-\frac\alpha2,2}}+\gamma^2
C|z_\gamma(t+r)|^2_{L^2}\Big]e^{C\int_r^s
|z_\gamma(t+\xi)|^{\frac{\alpha}{\alpha-1}}_{H^{s,
q}}d\xi}dr \nonumber \\
&\leq& e^{C\int_0^1 |z_\gamma(t+r)|^{\frac{\alpha}{\alpha-1}}_{H^{s,
q}}dr}\Big[|v_\gamma(t)|_{L^2}^{2}+ \int_0^1
C|z_\gamma(t+r)|^4_{H^{s,q}}dr \nonumber \\
&+&\int_0^1C|z_\gamma(t+r)|^4_{H^{1-\frac\alpha2,2}}dr +\gamma^2
C\int_0^1|z_\gamma(t+r)|^2_{L^2}dr \Big].\nonumber \\
\end{eqnarray}
Hence,
\begin{eqnarray*}
\frac1T\int_0^T &\mathbb{P}&\Big(\sup_{0\leq r\leq
1}\big|v_\gamma(t+r)\big|^{4}_{L^2}>\frac M{4c_\eta}
\Big)dt\nonumber \leq \frac1T\int_0^T \mathbb{P}\Big(e^{C\int_0^1
|z_\gamma(t+r)|^{\frac{\alpha}{\alpha-1}}_{H^{s, q}}dr}>\sqrt{\frac
M{4c_\eta}} \Big)dt\nonumber \\
& + & \frac1T\int_0^T
\mathbb{P}\Big(|v_\gamma(t)|_{L^2}^{2}>\sqrt{\frac M{4^3c_\eta}}
\Big)dt\nonumber \\  & + & \frac1T\int_0^T \mathbb{P}\Big(\int_0^1
|z_\gamma(t+r)|^4_{H^{s,q}}dr >\sqrt{\frac
M{4^3c_\eta C^2}} \Big)dt\nonumber \\
\end{eqnarray*}
\begin{eqnarray}\label{eq-est-P-T-squar}
{}& + &\frac1T\int_0^T
\mathbb{P}\Big(\int_0^1|z_\gamma(t+r)|^4_{H^{1-\frac\alpha2,2}}dr>\sqrt{\frac
M{4^3c_\eta C^2}} \Big)dt\nonumber \\
&+& \frac1T\int_0^T \mathbb{P}\Big(
\int_0^1|z_\gamma(t+r)|^2_{L^2}dr
>\sqrt{\frac
M{4^3c_\eta\gamma^4 C^2}} \Big)dt.\nonumber \\
\end{eqnarray}
Arguing as above and use Lemmata
\ref{lem-A-sigmaZ-gamma-epsilon-s-q},  \ref{lem-bound-mean} and
Chebychev inequality, then for enough large $ M $, we get
\begin{eqnarray}\label{eq-est-P-T-squar-part-1}
\frac1T\int_0^T \mathbb{P}\Big(\sup_{0\leq r\leq
1}\big|v_\gamma(t+r)\big|^{4}_{L^2}>\frac M{4c_\eta}
\Big)dt <\frac\eps4.\nonumber \\
\end{eqnarray}
Finally, from
\eqref{eq-est-sobolev-1}--\eqref{eq-est-P-T-term-eps-4} and
\eqref{eq-est-P-T-squar-part-1}, we get
\begin{eqnarray}
\frac1T\int_0^T\mathbb{P}\Big( |(-A)^\beta
v_\gamma(t+1)|^{2}_{L^2}>M \Big)dt <\eps.
\end{eqnarray}
Let us recall that $ u = v_\gamma+ z_\gamma $, hence using
\eqref{eq-est-P-T-squar-part-1} and  Lemma
\ref{lem-A-sigmaZ-gamma-epsilon-s-q}, we prove that
\begin{eqnarray}
\frac1T\int_0^T\mathbb{P}\Big( |(-A)^\beta u(t+1)|^{2}_{L^2}>M
\Big)dt <\eps.
\end{eqnarray}
Hence the family of probability measures $ \mu_T(\cdot):=
\frac1T\int_0^T(\mathcal{L}(X_{t+1}))(\cdot)dt $ is tight. By
Porokhorov's Theorem, we can subtract a weak convergent sequence
$\mu_{T_{n_k}}(\cdot), T_{n_k}\uparrow +\infty.$ Using
Krylov-Bogoliubov existence Theorem, see e.g.
\cite{DaPrato-Zabczyk-96}, we confirm the existence of at least one
invariant measure.

\subsection*{Proof of the Theorem \ref{Main-Theo}}
The proof of the main Theorem \ref{Main-Theo} follows, in a standard
way, as a conclusion of the above results. In fact, thanks to
Theorem \ref{Theo-Exist-invariant_measure}, there exists an
invariant measure. The semigroup $\mathbf U$ is strong Feller and
irreducible for all $ t>0$, hence by the Khas'minskii Theorem the
semigroup is regular for all $ t>0$ , see e.g. \cite[Proposition
4.1.1]{DaPrato-Zabczyk-96}. As a consequence of the regularity of
the semigroup and using Doob's Theorem, we conclude that the
invariant measure is unique. The convergence in \eqref{tv} follows
immediately from \cite[Theorem 1]{lukasz} and \cite[Proposition
4.2.1]{DaPrato-Zabczyk-96}, see also \cite{ms}. Finally, the
ergodicity of the invariant measure follows from the uniqueness, see
e.g. \cite[Theorem 3.2.6]{DaPrato-Zabczyk-96}.

\vspace{3mm}

\noindent{\large\bf Acknowledgements.} The research of the second
author is supported by The Austrian Science Foundation grant
P26017002.


\begin{thebibliography}{9}

\bibitem{Adams-Sbolev-spaces75} A. R. Adams, ``Sobolev Spaces'',
Academic Press New York, London 1975.

\bibitem{Albv-Flando-Sinai-08} S. Albeverio, F. Flandoli and Y. G. Sinai, `` SPDE in hydrodynamic: recent
progress and prospects'',  Lect. Not. in Math. 1942. Springer-Verlag
C.I.M.E. Florence 2008.

\bibitem{Barbato-Flandoli- Morandin-2010} D. Barbato,
F. Flandoli, and F. Morandin, ``Uniqueness for a stochastic inviscid
dyadic model'',  Proc. Amer. Math. Soc.  138  no. 7, 2010.

\bibitem{Biler-Funaki-98} P. Biler,  T. Funaki, and  W. A. Woyczynski, ``Fractal
Burgers equations'',  J. Differential Equations  148  no. 1, pp.
9-46, 1998.

\bibitem{Brz_1997} Z. Brze{\'z}niak, ``On stochastic convolution in Banach spaces and
applications'',  Stochastics Stochastics Rep.  61 no. 3-4, pp.
245-295, 1997.


\bibitem{Brz+Debbi_2007} Z. Brze{\'z}niak, and  L. Debbi, ``On stochastic Burgers equation
driven by a fractional Laplacian and space-time white noise'',
Stochastic differential equations: theory and applications, pp.
135-167, Interdiscip. Math. Sci., 2, World Sci. Publ., Hackensack,
NJ, 2007.

\bibitem{Brz-Gat-99}Z. Brze{\'z}niak,  and D. Gatarek, ``Martingale solutions and
invariant measures for stochastic evolution equations in Banach
spaces'',  Stochastic Process. Appl.  84  no. 2, pp. 187-225, 1999.


\bibitem{Brz+Li_2006} Z. Brze\'{z}niak, and Y. Li, ``Asymptotic compactness and absorbing sets for 2D stochastic
Navier-Stokes equations on some unbounded domains'', Trans. Am.
Math. Soc., N 358, pp.  5587-5629, 2006.

\bibitem{Brz+vN_2003} Z. Brze\'zniak and J.M.A.M. van Neerven, ``Space-time regularity for linear stochastic evolution
equations driven by spatially homogeneous noise'', J. Math. Kyoto
Univ. \textbf{43}, no. 2,  pp. 261-303,  2003.

\bibitem{Caffarelli-2009} L. A. Caffarelli, ``Some nonlinear problems involving non-local
diffusions'', ICIAM 07-6th Intern. Congress on Industrial and
Applied Math. pp. 43-56 Eur. Math. Soc. Zurich 2009.

\bibitem{Caffarelli-Vasseur2010} L. A. Caffarelli, and A. Vasseur, ``Drift diffusion equations with fractional diffusion and the
quasi-geostrophic equation'',  Ann. of Math. 2,  171  no. 3, pp.
1903-1930, 2010.


\bibitem{Constantin-chaae-wu-Arxiv10} D. Chae, P. Constantin, and  J. Wu ``Dissipative models generalizing the 2D Navier-Stokes and the surface
quasi-geostrophic equations'',  arXiv:1011.0171, 31 Oct 2010.

\bibitem{Crauel-Flandoli98} H. Crauel, and F. Flandoli, ``Hausdorff dimension of invariant sets for random
dynamical systems'',  J. Dynam. Differential Equations 10 no. 3, pp.
449-474, 1998.

\bibitem{DaPrato-Flandoli-2010} G. Da Prato, and  F. Flandoli, ``
Pathwise uniqueness for a class of SDE in Hilbert spaces and
applications'',  J. Funct. Anal. 259 no. 1, pp. 243-267, 2010.

\bibitem{DaPrato-Debussche-Differentiability-98} G. Da Prato, and A. Debussche, ``Differentiability of the transition semigroup of the stochastic
Burgers equation, and application to the corresponding
Hamilton-Jacobi equation'',  Atti Accad. Naz. Lincei Cl. Sci. Fis.
Mat. Natur. Rend. Lincei 9 Mat. Appl. 9  no. 4, pp. 267-277, 1998.


\bibitem{DapratoDebussche-Temam-94} G. Da Prato, A. Debussche, and R.
Temam, ``Stochastic Burgers' equation'',  NoDEA Nonlinear
Differential Equations Appl. 1 no. 4, pp. 389-402, 1994.

\bibitem{Daprato-Elworthy-zabczyk-95}G. Da Prato, K. D. Elworthy, and J. Zabczyk, `` Strong Feller property for
stochastic semilinear equations'', Stochastic Anal. Appl. 13 no. 1,
pp. 35-45, 1995.


\bibitem{DaPrato-Gatarek-95} G. Da Prato,  and  D. Gatarek,
``Stochastic Burgers equation with correlated noise'',  Stochastics
Stochastics Rep. 52  no. 1-2, pp. 29-41, 1995.



\bibitem{DaPrato-Zbc-92} G. Da Prato, and  J. Zabczyk, ``Stochastic equations in infinite
dimensions'',  Encyclopedia of Mathematics and its Applications 44.
Cambridge University Press  1992.


\bibitem{DaPrato-Zabczyk-96} G. Da Prato, and J. Zabczyk, ``Ergodicity for infinite-dimensional
systems'', London Mathematical Society Lecture Note Series 229.
Cambridge University Press, Cambridge 1996.


\bibitem{Dettw_1985} E. Dettweiler, ``Stochastic integration of Banach space valued functions'',
in Arnold L., Kotelenez P. (Eds.) {\sc Stochastic Space-Time Models
and Limit Theories}, D. Reidel, Dordrecht, pp. 53-79, 1985.


\bibitem{Flandoli-10} F. Flandoli, ``Random perturbation of PDEs and fluid dynamic
models'', Lectures Notes in Mathematics 2015, Ecole d'\'Et\'e de
Probabilit\'es de  Saint-Flour XL 2010.



\bibitem{Flandoli-etal-transportEq-10} F. Flandoli, and M. Gubinelli, and E.
Priola, ``Well-posedness of the transport equation by stochastic
perturbation'', Invent. Math. 180 no. 1, pp. 1-53, 2010.

\bibitem{Gatarek-Goldys-Inv-Meas-97} D. Gatarek, and B. Goldys, ``On invariant measures
for diffusions on Banach spaces'',  Potential Anal.  7  no. 2, pp.
539-553, 1997.

\bibitem{Goldys-Maslowski-05} B. Goldys, and B. Maslowski, ``Exponential ergodicity for stochastic Burgers and 2D Navier-Stokes
equations'',  J. Funct. Anal. 226  no. 1, pp. 230-255, 2005.

\bibitem{Hairer-MattinglyErgodicity-06} M. Hairer, and J. C. Mattingly, ``Ergodicity of the 2D
Navier-Stokes equations with degenerate stochastic forcing'', Ann.
of Math. (2) 164 no. 3, pp. 993-1032, 2006.


\bibitem{Kisleve-Regul-Blowup-Fract-ActiveScalars-Arxiv10} A. Kiselev, ``Regularity and blow up for active
scalars'', arXiv:1009.0540v1, 6 Sep 2010.


\bibitem{Kiselev-Nazarov-Fractal-Burgers-08} A. Kiselev, and F. Nazarov, and R. Shterenberg,`` Blow up and regularity for fractal Burgers
equation'',  Dyn. Partial Differ. Equ. 5 no. 3, pp. 211-240, 2008.


\bibitem{SugKak} T. Kukatani, and N. Sugimoto, `` Generalized {B}urgers Equations for Nonlinear
Viscoelastic Waves'',  Wave Motion 7, pp. 447-458, 1985.


\bibitem{LM-72-i} J. L. Lions, and E. Magenes, `` Non-Homogeneous Boundary
Value Problems and Applications'',  Vol. I English translation.
Springer-Verlag Berlin, New York 1972.


\bibitem{ms}
B. Maslowski, and J. Seidler, `` Probabilistic approach to the
strong Feller property'',  Probab. Theory Related Fields 118, No. 2,
pp. 187–210, 2000.


\bibitem{Neate-Truman08} A. D. Neate, and A. Truman, ``On the stochastic Burgers
equation and some applications to turbulence and astrophysics'',
Analysis and stochastics of growth processes and interface models,
pp. 281-305 Oxford Univ. Press, Oxford 2008.

\bibitem{Ondr_2004} M. Ondrej{\'a}t, ``Uniqueness for stochastic evolution equations
in {B}anach spaces'', {Dissertationes Math. (Rozprawy Mat.)} No 426,
2004.


\bibitem{Pablo-VazquezA-Arxiv} A. D. Pablo, F. Quiros, A. Rodriguez, and J. L. VazquezA, ``Fractional porous
medium equation'',  arXiv:1001.2383v1
[math.AP] 14 Jan 2010.

\bibitem{Pazy_1983} A. Pazy, ``Semigroups of linear operators and applications to partial
differential equations'',  Applied Mathematical Sciences 44.
Springer-Verlag, New York 1983.


\bibitem{Peszat-Zabc-strong-Feller-95} S. Peszat, and  J. Zabczyk,
`` Strong Feller property and irreducibility for diffusions on
Hilbert spaces'', Ann. Probab. 23 No. 1, pp. 157-172,  1995.

\bibitem{Rockner-BS-2006} M. R\"ockner, and  Z. Sobol, ``Kolmogorov equations in
infinite dimensions: well-posedness and regularity of solutions,
with applications to stochastic generalized Burgers  equations'',
Ann. Probab.  34  No. 2, pp. 663-727, 2006.

\bibitem{R&S-96} T. Runst, and  W. Sickel,`` Sobolev spaces of fractional order,
Nemytskij operators, and nonlinear partial differential equations'',
de Gruyter Series in Nonlinear Analysis and Applications 3. Walter
de Gruyter \& Co. Berlin 1996.

\bibitem{Shirikyan-06} A. Shirikyan, ``Ergodicity for a class of Markov processes and applications to
randomly forced PDE's'',  I.  Russ. J. Math. Phys.  12  No. 1, pp.
81-96, 2005.

\bibitem{lukasz} L. Stettner,`` Remarks on ergodic conditions for Markov
processes on Polish spaces'',  Bull. Polish Acad. Sci. Math. 42, No.
2, pp. 103–114, 1994.

\bibitem{Sug} N. Sugimoto,``Generalized Burgers equations and Fractional
Calculus'', Nonlinear Wave Motion.(A. Jeffery, Ed) pp. 162-179,
1991.

\bibitem{Temam-2001}R. Temam,`` Navier-Stokes equations. Theory and numerical
analysis'', Reprint of the 1984 edition. AMS Chelsea Publishing,
Providence, RI,  2001.


\bibitem{Triebel-interpolation-theo-78} H. Triebel, ``Interpolation theory, function spaces, differential
operators'', North-Holland Mathematical Library, 18. North-Holland
Publishing Co.,  Amsterdam-New York 1978.

\bibitem{Truman-Wu06} A. Truman, and  J. L. Wu, ``Fractal Burgers' equation driven by L\'evy
noise'', Lect. Notes Pure Appl. Math. 245 Chapman and Hall/CRC, Boca
Raton, FL 2006.


\bibitem{Neerven-2010}  J. Van Neerven, ``$\gamma$-radonifying operators-a survey'', The AMSI-ANU Workshop on
Spectral Theory and Harmonic Analysis, pp. 1-61, Proc. Centre Math.
Appl. Austral. Nat. Univ. 44 Austral. Nat. Univ. Canberra 2010.


\end{thebibliography}
\end{document}